\documentclass[leqno]{amsart}
\usepackage{adjustbox}
\usepackage{amsmath, amssymb, amsthm}
   \swapnumbers
   \makeatletter
     \def\swappedhead@plain#1#2#3{%
       \thmnumber{(\textup{#2})}
       \thmname{\@ifnotempty{#2}{~}\textup{#1}}
       \thmnote{ {\textup{(#3)}}}}
     \let\swappedhead\swappedhead@plain
   \makeatother
   \allowdisplaybreaks
   \theoremstyle{plain}
   \newtheorem{theo}[equation]{Theorem}
   \newtheorem{prop}[equation]{Proposition}
   \newtheorem{lemm}[equation]{Lemma}
   \newtheorem{coro}[equation]{Corollary}
   \theoremstyle{definition}
   \newtheorem{rema}[equation]{Remark}
   \newtheorem{scho}[equation]{Scholium}
	\newtheoremstyle{emptystyle}%
	  {1em plus .2em minus .1em}
	  {1em plus .2em minus .1em}
	  {\upshape}
	  {0pt}
	  {\bfseries}
	  {}
	  {.5em}
	  {\thmname{#1}\thmnumber{ #2}\textnormal{\thmnote{ #3.}}}
   \theoremstyle{emptystyle}
	\newtheorem{exam}[equation]{}
	\numberwithin{equation}{section}
 
\usepackage{bm}
\usepackage[inline]{enumitem}
	\setlist{topsep=1.5ex,label=\textup{(\alph*)},ref=\alph*}
\usepackage{epigraph}
\usepackage[mathscr]{eucal}
\usepackage[T1]{fontenc}               
\usepackage[utf8]{inputenc}            
   
   \let\dotlessi\i 
   
   \let\polishl\l
   \let\norwegiano\o
   
   \let\russianbreve\u
\usepackage{needspace}
\usepackage[new]{old-arrows}
   \renewcommand{\to}{\varto}
   \renewcommand{\mapsto}{\varmapsto}
   \renewcommand{\hookrightarrow}{\varhookrightarrow}

\usepackage{soulutf8}
\usepackage[dvipsnames]{xcolor}        
\usepackage{tikz-cd}
\usepackage[final,%
            bookmarksnumbered,%
            colorlinks=true,%
            linkcolor=RedViolet,%
            citecolor=RedViolet,%
            urlcolor=RoyalBlue]{hyperref}   

  \DeclareMathSymbol{A}{\mathalpha}{operators}{`A}
  \DeclareMathSymbol{B}{\mathalpha}{operators}{`B}
  \DeclareMathSymbol{C}{\mathalpha}{operators}{`C}
  \DeclareMathSymbol{D}{\mathalpha}{operators}{`D}
  \DeclareMathSymbol{E}{\mathalpha}{operators}{`E}
  \DeclareMathSymbol{F}{\mathalpha}{operators}{`F}
  \DeclareMathSymbol{G}{\mathalpha}{operators}{`G}
  \DeclareMathSymbol{H}{\mathalpha}{operators}{`H}
  \DeclareMathSymbol{I}{\mathalpha}{operators}{`I}
  \DeclareMathSymbol{J}{\mathalpha}{operators}{`J}
  \DeclareMathSymbol{K}{\mathalpha}{operators}{`K}
  \DeclareMathSymbol{L}{\mathalpha}{operators}{`L}
  \DeclareMathSymbol{M}{\mathalpha}{operators}{`M}
  \DeclareMathSymbol{N}{\mathalpha}{operators}{`N}
  \DeclareMathSymbol{O}{\mathalpha}{operators}{`O}
  \DeclareMathSymbol{P}{\mathalpha}{operators}{`P}
  \DeclareMathSymbol{Q}{\mathalpha}{operators}{`Q}
  \DeclareMathSymbol{R}{\mathalpha}{operators}{`R}
  \DeclareMathSymbol{S}{\mathalpha}{operators}{`S}
  \DeclareMathSymbol{T}{\mathalpha}{operators}{`T}
  \DeclareMathSymbol{U}{\mathalpha}{operators}{`U}
  \DeclareMathSymbol{V}{\mathalpha}{operators}{`V}
  \DeclareMathSymbol{W}{\mathalpha}{operators}{`W}
  \DeclareMathSymbol{X}{\mathalpha}{operators}{`X}
  \DeclareMathSymbol{Y}{\mathalpha}{operators}{`Y}
  \DeclareMathSymbol{Z}{\mathalpha}{operators}{`Z}

  \newcommand\NN{{\mathbf N}}        
  \newcommand\QQ{{\mathbf Q}}        
  \newcommand\RR{{\mathbf R}}        
  \newcommand\ZZ{{\mathbf Z}}        
  \newcommand\piG{\Gamma}            
  \newcommand\piH{\Delta}            
  \newcommand\piX{A}                 
  
  \newcommand\Lie{\mathfrak}
  \newcommand\LG{{\Lie{g}}}
  \newcommand\LH{{\Lie{h}}}

\DeclareMathOperator\ad{ad}
\DeclareMathOperator\Ad{Ad}
\DeclareMathOperator\GL{GL}
\DeclareMathOperator\id{id}
\DeclareMathOperator\INV{inv}
\DeclareMathOperator\Ker{Ker}
\DeclareMathOperator\SO{SO}
  \newcommand\BigWedge%
  {\mathord{\protect\adjustbox{valign=B,totalheight=.6\baselineskip}
  {$\bigwedge$}}}
  \newcommand\dCE{\textup{d}}
  \newcommand\e[1]{{\mathrm e^{\hspace{.06em}#1}}}			
\renewcommand\i{{\mathrm i}}										
  \newcommand\inv{^{-1}}											
\renewcommand\o{^\textup{o}}										

  \newcommand\<{\langle}											
\renewcommand\>{\rangle}											
  \newcommand\twoldots{\mathinner{\ldotp \ldotp}}

\theoremstyle{definition}
\newtheorem{defs}[equation]{Definitions \cite{Souriau:1985a,Iglesias-Zemmour:2013}}
\newtheorem{exes}[equation]{Examples \cite{Souriau:1985a,Iglesias-Zemmour:2013}}

\begin{document}

\title[Cohomology of a Lie group modulo a dense subgroup]{The de Rham cohomology of a Lie group\\ modulo a dense subgroup}

\author{Brant Clark}
\address{Department of Mathematical Sciences, Georgia Southern University, Statesboro, GA 30460, USA}
\curraddr{Department of Mathematics, University of Georgia, Athens, GA 30602, USA}
\email{brant.clark@uga.edu}

\author{François Ziegler}
\address{Department of Mathematical Sciences, Georgia Southern University, Statesboro, GA 30460, USA}
\email{fziegler@georgiasouthern.edu}

\date{December 9, 2025}
\subjclass[2020]{58A12, 57T15, 17B56, 58A40, 22E15}

\begin{abstract}
Let $H$ be a dense subgroup of a Lie group $G$ with Lie algebra $\LG$. We show that the (diffeological) de Rham cohomology of $G/H$ equals the Lie algebra cohomology of $\LG/\LH$, where $\LH$ is the ideal \mbox{$\{Z\in\LG:\exp(tZ)\in H \text{ for all } t\in\RR\}$}.
\end{abstract}

\maketitle

\setlength\epigraphwidth{.5\textwidth}
\renewcommand{\epigraphsize}{\footnotesize}
\setlength\epigraphrule{0pt}
\epigraph{\hfill\so{C'est un problème d'algèbre}.}{---É.~Cartan \cite[p.\,392]{Cartan:1937d}}

\setcounter{section}{-1}
\section{Introduction}

When $H$ is a non-closed subgroup of a Lie group $G$, the ingrained habit of giving subsets the subset topology and quotients the quotient topology is rather sterile. Indeed in that case $H$'s subset topology is not a Lie group topology, and~$G/H$'s~quotient topology is not Hausdorff: it is even trivial when (and only when) $H$ is dense. (See the Appendix, \eqref{closed_subgroup}, for references to these and other facts.) The subgroup situation was completely clarified by Bourbaki \eqref{initial_subgroup_theorem}: \emph{any} subgroup $H$ is canonically a Lie group, with possibly-finer-than-subset topology and Lie algebra
\begin{equation}
	\label{subalgebra}
	\LH = \{Z\in\LG:\exp(tZ)\in H \text{ for all } t\in\RR\}.
\end{equation}
The subtler case of quotients has given rise to several approaches. A leading one \cite{Connes:1985,Loday:1986} is to seek ``non-commutative topology'' \cite{Effros:1979} in the (periodic cyclic) cohomology of a crossed product algebra, used in place of the deficient (e.g.~trivial) commutative algebra of smooth functions on $G/H$.

In this paper we explore concurrent ideas of Souriau \emph{et al.}~\cite{Souriau:1985a,Donato:1985a,Donato:1987}, who embedded \{manifolds\} into the larger category \{diffeological spaces\} which has 1$^\circ$) arbitrary subobjects and quotient objects, yet still 2$^\circ$) on each, a de Rham complex $(\Omega^\bullet(X),d)$ and resulting cohomology $H_{\textup{dR}}^\bullet(X)$. We find that these are very simple:

\begin{theo}
	\label{main_theorem}
	Let $H$ be a dense subgroup of a Lie group $G$. Then \eqref{subalgebra} is an~ideal in $\LG$\textup, and we have an isomorphism $(\Omega^\bullet(G/H),d)=(\BigWedge{}^\bullet(\LG/\LH)^*,\dCE)$ and hence
	\begin{equation}
		\label{isomorphism}
		H_{\textup{dR}}^\bullet(G/H)=H^\bullet(\LG/\LH),
	\end{equation}
	where the right-hand sides are the Chevalley--Eilenberg complex of $\LG/\LH$ and its coho\-mology \cite[§14]{Chevalley:1948}.
\end{theo}

Note that this is doubly unusual: first, manifolds never have \mbox{$0<\dim\Omega^\bullet(X)<\infty$}; secondly, known $H^\bullet(\LG/\LH)$ will yield nonzero $H_{\textup{dR}}^\bullet(G/H)$ \emph{despite} the trivial topology. Towards explicit examples, it helps to specialize \eqref{main_theorem} to the extreme cases where $H$ is either \emph{D-connected} or \emph{D-discrete}, where for reasons to appear in (\ref{Diffeology_definitions}\ref{D-topology}, \ref{Diffeology_examples}\ref{manifold_diffeology}), `D-' means the Lie group topology mentioned before \eqref{subalgebra}; equivalently, see \eqref{Yamabe}, these two cases correspond to the subset topology of $H$ being either arcwise connected or totally arcwise disconnected (i.e.~arc components are points). We will obtain:

\begin{coro}
	\label{corollary}
	If the dense subgroup $H\subset G$ is \textup{(a)} D-connected\textup, \textup{(b)} D-discrete\textup, or \textup{(c)} a D-discrete additive subgroup $\piX$ in a vector space $V$\textup, then we have respectively
	\begin{enumerate}[itemsep=1.5ex,itemindent=-1pt]

		\item
		\label{D-connected}
		$H_{\textup{dR}}^\bullet(G/H)=\BigWedge{}^\bullet(\LG/\LH)^*$
		\textup(full exterior algebra\textup{);}

		\item
		\label{D-discrete}
		$H_{\textup{dR}}^\bullet(G/H)=H^\bullet(\LG)$
		\textup(and all Lie algebra cohomology rings occur in this~way\textup{);}

		\item
		\label{vector}
		$H_{\textup{dR}}^\bullet(V/\piX)=\BigWedge{}^\bullet V^*$
		\textup(full exterior algebra\textup).
	\end{enumerate}
	Moreover the resemblance of \eqref{D-connected} to \eqref{vector} is no accident\textup, for $G/H$ in \eqref{D-connected} can always~be rewritten as a \emph{quasitorus} \textup(\cite{Prato:2001,Iglesias-Zemmour:2021,Karshon:2025}\textup) $V/\piX$ as in \eqref{vector}\textup, with $V=\LG/\LH$.
\end{coro}

The classic example of \eqref{D-connected} is the $2$-torus $G=\smash[b]{\left(\begin{smallmatrix} S^{\smash1}&\hspace{-1pt}0^{\smash{\phantom1}}
\\
0^{\phantom1}&\hspace{-1pt}S^1
\end{smallmatrix}\hspace{-3pt}\right)}$ modulo an \mbox{irrational winding}
\begin{equation}
	\label{Kronecker}
	H=
	\left\{
	\begin{pmatrix}
		\e{2\pi\i t}&0\\
		0&\e{2\pi\i\alpha t}
	\end{pmatrix}:t\in \RR
	\right\}
	\rlap{\qquad\qquad$(\alpha\notin\QQ)$.}
\end{equation}
Here each coset of $H$ meets the transversal $\dot G=\smash[b]{\left(\begin{smallmatrix} 1&0^{\smash{\phantom1}}\\0&S^1 \end{smallmatrix}\!\right)}$ in a coset of $\dot H=\smash[b]{\left(\begin{smallmatrix} 1&0\\0&\e{2\pi\i\alpha\ZZ} \end{smallmatrix}\right)}$, so the rewriting is $G/H=\dot G/\dot H=\RR/(\ZZ+\alpha\ZZ)$, a \emph{quasicircle}, which (\ref{D-connected},\,\ref{D-discrete},\,\ref{vector}) all~agree has de Rham cohomology $\BigWedge^\bullet\RR=\RR\oplus\RR$, the same as a circle. This was observed in \cite[2.4]{Donato:1987}; later (\ref{corollary}\ref{vector}) appeared in \cite[Ex.\,105]{Iglesias-Zemmour:2013}, and \eqref{main_theorem} for $1$-forms in \cite[\nolinebreak 9.14]{Barbieri:2026}. Lastly, upon announcing \eqref{main_theorem} itself we learned from H.~Kihara that he independently obtained it in his forthcoming \cite{Kihara:2026}.

If we allow non-dense but \emph{closed} subgroups, then equalities like \eqref{isomorphism} have of course a much longer history, told in \cite[pp.\,84--85, 152--153]{Borel:2001a} and going back to Cartan's papers \cite{Cartan:1928,Cartan:1930a} which inspired both de Rham and Chevalley--Eilenberg. Ultimately his results (for $H=\{e\}$ or $G/H$ symmetric) were generalized into the following: \emph{whenever $G$ is compact connected and $H$ closed connected\textup, one has}
\begin{equation}
	\label{relative}
	H_{\textup{dR}}^\bullet(G/H)=H^\bullet(\LG,\LH)
\end{equation}
\cite[§22]{Chevalley:1948}. Here the right-hand side is relative Lie algebra cohomology, of which we will only need to know that it boils down to $H^\bullet(\LG/\LH)$ when $\LH$ is an ideal, as in \eqref{main_theorem}. So a common generalization of \eqref{main_theorem} and \eqref{relative} to arbitrary subgroups seems next in order. We do not achieve it here, however,
as both reduce matters to $G$-invariant forms (and thereby algebra) 
too differently: in \eqref{main_theorem} it happens at the cochain level and by \emph{density}; in \eqref{relative} it happens only in cohomology and by \emph{averaging}, which essentially requires $G$ compact.

Our plan below is to review the diffeological setting in §§\ref{vocabulary}--\ref{complex}, then prove \eqref{main_theorem} and \eqref{corollary} in §§\ref{forms}--\ref{proofs}, and add examples in  §\ref{examples}. Appendix \ref{subgroups} collects known subgroup properties with which we didn't wish to clutter the exposition, and Appendix \ref{proof_CE} contributes a quick proof of the Chevalley--Eilenberg coboundary formula \eqref{Chevalley-Eilenberg}.

\section{Diffeological vocabulary}\label{vocabulary}

Suppose $X$ is a manifold, and write $\tau_m$ for the Euclidean topology of $\RR^m$. Then $\mathscr P := \bigcup_{m\in\NN,\,U\in\tau_m}\!C^\infty(U,X)$ satisfies the following, where `is a plot' means `$\in\mathscr P\,$':

\Needspace*{5\baselineskip}
\begin{enumerate}[label=\textup{(D\arabic*)},ref=D\arabic*]

	\item
	\label{D1}\emph{Covering.}
	All constant maps $\RR^m\to X$ are plots, for all $m$.

	\item
	\label{D2}\emph{Locality.}
	Let $U\overset{P}{\to}X$ be a map with $U\in\tau_m$. If every point of $U$ has an open neighborhood $V$ such that $P_{|V}$ is a plot, then $P$ is a plot.

	\item
	\label{D3}\emph{Smooth compatibility.} 
	Let $U\overset{\Phi}{\to}V\overset{Q\,\,}{\to}X$ be maps with $(U,V)\in\allowbreak\tau_m\times\tau_n$. If~$Q$ is a plot and $\Phi\in C^\infty(U,V)$, then $Q\circ\Phi$ is a plot.
\end{enumerate}
Diffeology \cite{Souriau:1985a,Iglesias-Zemmour:2013} is a generalization of manifold theory, where instead of declaring which maps $U\to X$ are diffeomorphisms (`charts'), one declares which are merely smooth (`plots'), subject to (\ref{D1}--\ref{D3}) as axioms --- of which interestingly, (\ref{D1}--\ref{D2}) were already in \cite[1.19\emph{b\textup,e}]{Souriau:1958}. In more detail, writing $\operatorname{Maps}(U,X)=X^U$:

\begin{defs} \ 
	\label{Diffeology_definitions}
	\begin{enumerate}

		\item
		Let $X$ be a set. A \emph{diffeology} on $X$ is a subset $\mathscr P$ of \,$\bigcup_{m\in\NN,\,U\in\tau_m}\!\operatorname{Maps}(U,X)$ satisfying (\ref{D1}--\ref{D3}); members of $\mathscr P$ are called \emph{plots}.

		\item
		\label{smooth_map}
		A map $(X,\mathscr P)\overset{F}{\to} (Y,\mathscr Q)$ between diffeological spaces (:~sets with diffeologies) is called \emph{smooth} if $P\in\mathscr P$ implies $F\circ P\in\mathscr Q$.

		\item
		\label{D-topology}
		\vphantom{$\overset{F}{\to}$}%
		A subset of a diffeological space is \emph{D-open}, and a member of the \emph{D-topology}, if its preimage by every plot is Euclidean open.

		\item
		If $(X,\mathscr P)\overset{\id}{\to}(X,\mathscr Q)$ is smooth, i.e.~$\mathscr P\subset\mathscr Q$, we call $\mathscr P$ \emph{finer} and $\mathscr Q$ \emph{coarser}.
	\end{enumerate}
With this, diffeological spaces and smooth maps make a category; (\ref{D-topology}) defines a functor from it to topological spaces and continuous maps; and every diffeology $\mathscr P$ sits between  $\mathscr P_{\text{discrete}} =$ \{locally constant maps\} and $\mathscr P_{\text{coarse}}=$ \{all maps\}.
\end{defs}

\begin{exes} \
	\label{Diffeology_examples}
	\begin{enumerate}

		\item
		\label{manifold_diffeology}
		What we said before (\ref{D1}--\ref{D3}) endows every manifold with a canonical \emph{manifold diffeology}. We say that a diffeological space \emph{is a manifold} if it can be so obtained; then (\ref{Diffeology_definitions}\ref{smooth_map},\,\ref{D-topology}) boil down to the ordinary notions.

		\item
		\label{subset_diffeology}
		Let $Y$ be a diffeological space and $i:X\to Y$ an injection. Then $X$ has a coarsest diffeology making $i$ smooth, the \emph{subset diffeology}, characterized by: $F:Z\to X$ (from another diffeological space) is smooth iff $i\circ F$ is smooth. Its plots are the maps $P:U\to X$ such that $i\circ P$ is a plot of $Y$:
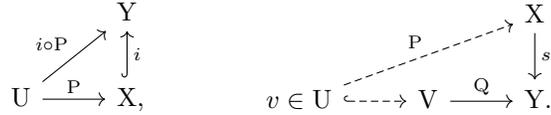
		\begin{equation*}
			\hspace{1cm}
			\begin{tikzcd}
				& Y
				\\
				  U
				  \ar[r,"P"]
				  \ar[ur,"i\circ P"]
				& X
				  \ar[u,hook,"i"']
				& Z\rlap{.}
				  \ar[l,"F"']
				  \ar[ul,"i\circ F"']
			\end{tikzcd}
		\end{equation*}

		\item
		\label{quotient_diffeology}
		Let $X$ be a diffeological space and $s:X\to Y$ a surjection. Then $Y$ has a finest diffeology making $s$ smooth, the \emph{quotient diffeology}, characterized~by: $F:Y\to Z$ (to another diffeological space) is smooth iff $F\circ s$ is smooth. Its~plots are the maps $Q:V\to Y$ that have around each $v\in V$ a `local lift': a plot $P:U\to X$ with $U\subset V$ an open neighborhood of $v$ and $s\circ P=Q_{|U}$:
		\begin{equation*}
			\hspace{-1.2cm}
			\begin{tikzcd}
				&
				& X
				  \ar[d,"s"]
				  \ar[dr,"F\circ s"]
				\\
				  v\in U
				  \ar[r,dashed,hook]
				  \ar[urr,dashed,"P"]
				& V
				  \ar[r,"Q"]
				& Y
				  \ar[r,"F"]
				& Z\rlap{.}
			\end{tikzcd}
		\end{equation*}

		\item
		\label{caution}
		Caution: while quotient diffeologies have D-topology = quotient topology, subset diffeologies generally have D-topology $\supset$ subset topology.
	\end{enumerate}
\end{exes}

\section{Diffeological de Rham complex}\label{complex}

Let us agree to call \emph{ordinary} the $k$-forms on Euclidean open sets $V\subset\RR^n$ or manifolds $X$, and operations on them (exterior derivative $d$, pull-back $\Phi^*$). Thus an ordinary $k$-form $\omega\in\Omega^k(V)$ is a smooth map $v\mapsto\omega_v$ from $V$ to the space $\smash{\BigWedge^k(\RR^n)^*\cong\RR^{\binom n k}}$ of alternating $k$-linear maps $\RR^n\times\dots\times\RR^n\to\RR$ ($k$ factors), and its ordinary exterior derivative and pull-back by $\Phi\in C^\infty(U,V)$ are given by
\begin{equation}
	\label{ordinary_d}
	(d\omega)_v(v_0,\dots,v_k)
	=\sum_{i=0}^k(-1)^i\frac{\partial\omega_v}{\partial v}(v_i)(v_0,\dots,\widehat{v_i},\dots,v_k)
\end{equation}
(hat means `omit') and $(\Phi^*\omega)_u(u_1,\dots,u_k) = \omega_{\Phi(u)}(D\Phi(u)(u_1),\dots,D\Phi(u)(u_k))$.

\begin{defs}
	\label{De_Rham_definitions}
	Let $X$ and $Y$ be diffeological spaces.
	\begin{enumerate}

		\item
		\label{k-form}
		A (diffeological) \emph{$k$-form on} $Y$ is a functional $\beta$ which sends each plot $Q:V\to Y$ to an ordinary $k$-form on $V$, \emph{denoted} $Q^\star\beta$ (note special $\star$). As compatibility, we require: if $\Phi\in C^\infty(U,V)$ (so $Q\circ\Phi$ is another plot), then
		\begin{equation*}
			(Q\circ\Phi)^\star\beta = \Phi^*Q^\star\beta,
			\qquad
			\Phi^* : \text{ordinary pull-back}.
		\end{equation*}

		\item
		\label{pull-back}
		Its \emph{pull-back} by a smooth map $F:X\to Y$ is the $k$-form $F^*\beta$ on $X$ defined by: if $P$ is a plot of $X$ (so $F\circ P$ is a plot of $Y$), then
		\begin{equation*}
			P^\star F^*\beta = (F\circ P)^\star\beta,
			\qquad
			F^* : \text{being defined}.\hspace{2em}
		\end{equation*}

		\item
		\label{exterior_derivative}
		Its \emph{exterior derivative} is the $(k+1)$-form $d\beta$ on $Y$ defined by: if $Q$ is a plot of $Y$, then $Q^\star d\beta = d [Q^\star\beta]$, with ordinary $d$ on the right-hand side.
	\end{enumerate}
	The \emph{de Rham complex} $(\Omega^\bullet(Y),d)$ is the sum of the spaces $\Omega^k(Y)$ of $k$-forms on $Y$, endowed with the differential (\ref{exterior_derivative}), which satisfies $d^{\mspace{1mu}2}=0$ because \eqref{ordinary_d} does. Its cohomology is the \emph{de Rham cohomology} $H_{\text{dR}}^\bullet(Y)$.
\end{defs}

\begin{scho}
	One checks without trouble that (\ref{k-form},\,\ref{pull-back},\,\ref{exterior_derivative}) above imply the following, which hold true for all $k$-forms $\beta$ and smooth maps $F, G$ \cite{Souriau:1985a,Iglesias-Zemmour:2013}:
	\begin{equation}
		\label{pull_back_commutes}
		(F\circ G)^*\beta = G^*F^*\beta,
		\qquad\qquad
		d[F^*\beta] = F^*d\beta.
	\end{equation}
	If $Y$ is an Euclidean open set, each diffeological $k$-form $\beta$ on $Y$ defines an ordinary one, $b=\id_Y^\star\beta$, and (\ref{k-form}) (applied with $\id_Y,Q$ in place of $Q,\Phi$) forces $Q^\star\beta$ to always equal the ordinary pull-back $Q^*b$. Likewise if $Y$ is a manifold (\ref{Diffeology_examples}\ref{manifold_diffeology}), then (\ref{k-form}) (applied at first to \emph{charts} $V\to Y$) ensures that there is an ordinary $k$-form $b$ (:~section of $\smash{\BigWedge^kT^*Y}$) such that $Q^\star\beta$ and $Q^\star d\beta$ are always just the ordinary $Q^*b$ and $Q^*db$. So on manifolds we may (and will) suppress the distinction between diffeological and ordinary $k$-forms and operations on them; hence we retire the special $\star$, and (\ref{De_Rham_definitions}\ref{k-form},\,\ref{pull-back},\,\ref{exterior_derivative}) become special cases of \eqref{pull_back_commutes}.
\end{scho}

One often needs to decide if a given $k$-form is pulled back from a quotient. For this we have the following criterion, proved in \cite[2.5]{Souriau:1985a} or \cite[6.38--39]{Iglesias-Zemmour:2013}:

\begin{prop}
	\label{pullback_by_subduction}
	Let $s:X\to Y$ be a subduction between diffeological spaces\textup, i.e.~a smooth surjection such that $Y$ has precisely the quotient diffeology \textup{(\ref{Diffeology_examples}\ref{quotient_diffeology})}. Let $\alpha\in\Omega^k(X)$. Then $\alpha=s^*\beta$ for some $\beta\in\Omega^k(Y)$ iff all plots $P, Q$ of $X$ satisfy\textup:
	\begin{equation}
		\label{criterion}
		s\circ P=s\circ Q
		\qquad\Rightarrow\qquad
		P^*\alpha = Q^*\alpha.
	\end{equation}
	Moreover $\beta$ is then unique\textup, i.e.\textup, pull-back $s^*:\Omega^k(Y)\to\Omega^k(X)$ is injective.\qed
\end{prop}

\section{Differential forms on $G/H$ ($H$ dense)}\label{forms}

We now assume our theorem's hypotheses: $G$ is a Lie group, $H$ a dense subgroup (i.e.~$H$ meets every open subset of $G$). Endow $X=G/H$ with the quotient diffeology, and write $\Pi:G\to X$ for the natural projection, $\Pi(q)=qH$. Also write $L_g$ and $R_g:G\to G$ for left and right translation by $g\in G$: $L_g(q) = gq$ and $R_g(q) = qg$.

\begin{prop}
	\label{pull_backs_to_G}
	Pull-back via $\Pi$ defines a bijection $\Pi^*$ from $\Omega^k(X)$ onto the set of those $\mu\in\Omega^k(G)$ that are
	\begin{enumerate}[topsep=1ex,itemsep=1ex]

		\item
		\label{right-invariant}\emph{right-invariant:} 
		$R_g^*\mu=\mu$ for all $g\in G$\textup;

		\item
		\label{horizontal}\emph{$\LH$-horizontal:}
		$\mu(Z_1,\dots,Z_k)=0$ whenever one of the $Z_j\in\LG$ is in $\LH$.
	\end{enumerate}
\end{prop}

\begin{proof}
	Let us first note that since $H$ is dense, we know from (\ref{vanEst}\ref{invariant}) that
	\begin{equation}
		\label{G_normalizes_h}
		\text{$G$ normalizes $\LH$:\qquad}
		\Ad(g)(\LH) = \LH
		\quad\text{for all }
		g\in G.
	\end{equation}
	Suppose $\mu=\Pi^*\alpha$ for some $\alpha\in\Omega^k(X)$. We must prove (\ref{right-invariant}) and (\ref{horizontal}). Now the relation $\Pi\circ R_h=\Pi$ implies $R_h^*\Pi^*\alpha=\Pi^*\alpha$ for all $h\in H$ \eqref{pull_back_commutes}, and since $H$ is dense, the same follows for all $g\in G$: so $\mu$ is right-invariant. To see that it is $\LH$-horizontal, consider the two plots $P, Q:\LG\times\LH\to G$ sending $u=(Z,W)$ to
	\begin{equation}
		P(u)=\exp(Z),
		\qquad\text{resp.}\qquad
		Q(u)=\exp(Z)\exp(W).
	\end{equation}
	(For these to be literally plots, use bases to identify $U:=\LG\times\LH$ with some $\RR^m$.) Then clearly $\Pi\circ P=\Pi\circ Q$, so by the criterion \eqref{pullback_by_subduction} we have $P^*\mu=Q^*\mu$. As $DP(0,0)$ and $DQ(0,0)$ map $(Z,W)$ respectively to $Z$ and $Z+W$, this implies $\mu(Z_1,\dots,Z_k)=\mu(Z_1+W_1,\dots,Z_k+W_k)$ for all choices of $(Z_i,W_i)\in T_{(0,0)}U$. If $Z_j\in\LH$, then choosing $W_j=-Z_j$ yields (\ref{horizontal}).
	
	Conversely, suppose that $\mu\in\Omega^k(G)$ satisfies (\ref{right-invariant}) and (\ref{horizontal}), and let $P,Q:U\to G$ be any two plots with $\Pi\circ P=\Pi\circ Q$. We must show that $P^*\mu = Q^*\mu$. Now $\Pi\circ P=\Pi\circ Q$ means that $R(u):= P(u)\inv Q(u)$ defines a plot $R:U\to H$. Thus we have an ordinary smooth map $P\times Q\times R$ sending $u\in U$ to
	\begin{equation}
		(g,gh,h):=(P(u), Q(u), R(u)).
	\end{equation}
	Its derivative at $u$ will send each $\delta u\in T_uU$ to a tangent vector we choose to denote $(\delta g,\delta[gh], \delta h)\in T_gG\times T_{gh}G\times T_hH$. Now following \cite[III.2.2]{Bourbaki:1972}, write simply $g.v$ and $v.g$ for the images of a vector $v\in T_qG$ under the derivatives $DL_g(q)$ and $DR_g(q)$. Then we have $\Ad(g)(Z) = g.Z.g\inv$ and $\delta[gh] = \delta g.h + g.\delta h$, whence, given $k$ tangent vectors $\delta_1u,\dots,\delta_ku\in T_uU$,
	\begin{equation}
		\label{little_computation}
		\begin{aligned}
			\delta_i[gh].(gh)\inv
			&=[\delta_ig.h + g.\delta_ih].(gh)\inv\\
			&=\delta_ig.g\inv + \Ad(g)(\delta_ih.h\inv).
		\end{aligned}
	\end{equation}
	By \eqref{G_normalizes_h}, the second term here (call it $W_i$) is in $\LH$. Thus we obtain
	\begin{align}
		(Q^*\mu)(\delta_1u,\dots,\delta_ku)
		&=\mu(\delta_1[gh],\dots,\delta_k[gh])\notag\\
		&=\mu(\delta_1[gh].(gh)\inv,\dots,\delta_k[gh].(gh)\inv)
		&&\text{by (\ref{right-invariant})}\notag\\
		&=\mu(\delta_1g.g\inv + W_1,\dots,\delta_kg.g\inv + W_k)
		&&\text{by \eqref{little_computation}}\notag\\
		&=\mu(\delta_1g.g\inv,\dots,\delta_kg.g\inv)
		&&\text{by (\ref{horizontal})}\\
		&=\mu(\delta_1g,\dots,\delta_kg)
		&&\text{by (\ref{right-invariant})}\notag\\
		&=(P^*\mu)(\delta_1u,\dots,\delta_ku)&&\notag
	\end{align}
	as desired. So \eqref{pullback_by_subduction} says that $\mu$ is in the image of the injection $\Pi^*$, and the proof is complete.
\end{proof}

\section{End of proofs}\label{proofs}
Lie algebra cohomology is traditionally defined (or motivated) using \emph{left-} rather than right-invariant forms on $G$. To switch between the two, we need only pull back by the inversion map $\INV:g\mapsto g\inv$. Indeed the relation $\INV\circ\,L_g = R_{g\inv}\circ\INV$ readily implies that  $\mu\in\Omega^k(G)$ is right-invariant iff $\omega=\INV^*\mu$ is left-invariant. Also $\INV^*$ preserves $\LH$-horizontality, because the derivative of $g\mapsto g\inv$ at $e$ is $Z\mapsto -Z$. So \eqref{pull_backs_to_G} gives:

\begin{coro}
	\label{inverted_pull_backs}
	In the setting of \eqref{pull_backs_to_G}\textup, pull-back via $\check\Pi=\Pi\circ\INV$ defines a bijection $\check\Pi^*=\INV^*\Pi^*$ from $\Omega^k(X)$ onto the set of those $\omega\in\Omega^k(G)$ that are
	\begin{enumerate}[topsep=1ex,itemsep=1ex]
		
		\item
		\label{left-invariant}\emph{left-invariant:} 
		$L_g^*\omega=\omega$ for all $g\in G$\textup;
		
		\item
		\label{h-horizontal}\emph{$\LH$-horizontal:}
		$\omega(Z_1,\dots,Z_k)=0$ whenever one of the $Z_j\in\LG$ is in $\LH$.\qed
		
	\end{enumerate}
\end{coro}

Now, left-invariant forms (\ref{inverted_pull_backs}\ref{left-invariant}) make a subcomplex $(\Omega^\bullet(G)^G,d)$ of $(\Omega^\bullet(G),d)$ which depends only on $\LG$: for they satisfy, for all $Z_i\in\LG$, the relation (notation \ref{little_computation}) $\omega(g.Z_1,\dots,g.Z_k) = \omega(Z_1,\dots,Z_k)$ which characterizes $\omega$ by its value $\omega_e\in\smash{\BigWedge{}^k\LG^*}$, and the \emph{Chevalley--Eilenberg formula}
\begin{equation}
	\label{Chevalley-Eilenberg}
	d\omega(Z_0,\dots,Z_k) = \!\!\sum_{0\leqslant i<j\leqslant k}\!\!
	(-1)^{i+j}\omega([Z_i,Z_j],Z_0,\twoldots,\widehat Z_i,\twoldots,\widehat Z_j,\twoldots,Z_k)
\end{equation}
which computes $(d\omega)_e$ from $\omega_e$ alone. Thus, using \eqref{Chevalley-Eilenberg} as definition of a coboundary $\dCE$ on $\BigWedge{}^\bullet\LG^*$, we obtain a complex $(\BigWedge{}^\bullet\LG^*,\dCE)$ isomorphic to $(\Omega^\bullet(G)^G,d)$ via $\omega\mapsto\omega_e$. Its cohomology is by definition the \emph{Lie algebra cohomology} $H^\bullet(\LG)$. (For all this see for instance \cite[III.3.14]{Bourbaki:1972} or \cite[14.14\,\emph{sq}]{Michor:2008}, in addition to \cite[§9, §14]{Chevalley:1948} whose normalizations differed slightly.)

Our interest, however, lies in the further subcomplex $\smash{\Omega^\bullet(G)^G_\LH}$ of forms that satisfy also (\ref{inverted_pull_backs}\ref{h-horizontal}); or equivalently via $\omega\mapsto\omega_e$, its isomorph $\smash{(\BigWedge{}^\bullet\LG^*)_\LH}$ defined by (\ref{inverted_pull_backs}\ref{h-horizontal}) inside $\BigWedge{}^\bullet\LG^*$. For this we have the following, where $\pi:\LG\to\LG/\LH$ is the natural projection.

\begin{lemm}
	\label{induced_morphism_of_complexes}
	We have $[\LG,\LH]\subset\LH$\textup, i.e.~$\LH$ is an ideal\textup, and pull-back via $\pi$ defines an isomorphism $\pi^*$ of $(\BigWedge{}^\bullet(\LG/\LH)^*,\dCE)$ onto the subcomplex $((\BigWedge{}^\bullet\LG^*)_\LH,\dCE)$ of $(\BigWedge{}^\bullet\LG^*,\dCE)$.
\end{lemm}

\begin{proof}
	Deriving \eqref{G_normalizes_h} at $e$ gives $[\LG,\LH]\subset\LH$. The rest is functoriality of $\BigWedge{}^\bullet(\,\cdot\,)^*$ and essentially the end remark of \cite[\nolinebreak §22]{Chevalley:1948}; we sketch the elementary argument. Pull-back $\pi^*: \smash{\BigWedge^k(\LG/\LH)^*} \to \smash{\BigWedge^k\LG^*}$ is defined by $(\pi^*\sigma)(Z_1,\dots,Z_k) = \sigma(\pi(Z_1),\dots,\pi(Z_k))$. It is one-to-one because $\pi$ is onto, and clearly $\pi^*\sigma$ is always $\LH$-horizontal. Conversely if $\omega$ is $\LH$-horizontal, it is $\pi^*\sigma$ with $\sigma(A_1,\dots,A_k):=\omega(Z_1,\dots,Z_k)$ where $Z_i$ is any member of $\pi\inv(A_i)$: (\ref{inverted_pull_backs}\ref{h-horizontal}) ensures that this is well-defined. So we get a linear bijection $\BigWedge{}^\bullet(\LG/\LH)^* \to (\BigWedge{}^\bullet\LG^*)_\LH$, which commutes with $\dCE$ because of \eqref{Chevalley-Eilenberg} and $\pi([Z_i,Z_j]) = [\pi(Z_i),\pi(Z_j)]$.
\end{proof}

\begin{proof}[Proof of \eqref{main_theorem}]
	The theorem now follows by composing the three isomorphisms of complexes seen in (\ref{inverted_pull_backs}, \ref{Chevalley-Eilenberg}, \ref{induced_morphism_of_complexes}):
	\begin{equation}
		\begin{tikzcd}[row sep=large,column sep=large,every label/.append style={font=\small}]
			\Omega^\bullet(G)^G_{\smash{\LH}}
			\ar[rr,"\omega\mapsto\omega_e"]
			&&
			(\BigWedge{}^\bullet\LG^*)_\LH
			\\
			\Omega^\bullet(X)
			\ar[u,"\check\Pi^*"]
			\ar[rr,dashed] 
			&&
			\BigWedge{}^\bullet(\LG/\LH)^*\rlap{.}
			\ar[u,swap,"\pi^*"]
		\end{tikzcd} 
	\end{equation} 
	Of these the first uses, of course, the commutativity \eqref{pull_back_commutes} of $d$ with $\check\Pi^*$.
\end{proof}

\begin{proof}[Proof of \textup{(\ref{corollary})}]
	\eqref{D-connected} If $H$ is D-connected, we know that $\LG/\LH$ is \emph{abelian} (\ref{vanEst}\ref{abelian}). So all coboundaries in $\BigWedge{}^\bullet(\LG/\LH)^*$ vanish \eqref{Chevalley-Eilenberg}, and  \eqref{isomorphism} is this full exterior algebra.

	\eqref{D-discrete} If $H$ is  D-discrete, we have $\LH=\{0\}$, so  \eqref{isomorphism} says that $H_{\textup{dR}}^\bullet(G/H)=H^\bullet(\LG)$. Every Lie algebra cohomology ring $H^\bullet(\LG)$ occurs in this way, for given $\LG$ we can find a connected Lie group $G$ with Lie algebra $\LG$, and then in $G$ always a countable dense subgroups $H$ \cite[4.2]{Gelander:2017}, which is D-discrete by \cite[Ex.\,8]{Iglesias-Zemmour:2013}.

\eqref{vector} If $G/H$ is the quotient $V/\piX$ of a vector space by a D-discrete dense additive subgroup, then again \eqref{Chevalley-Eilenberg} is zero and (\ref{corollary}\ref{D-discrete}) is the full exterior algebra $\BigWedge{}^\bullet V^*$.

	Finally, to see how case \eqref{D-connected} always boils down to \eqref{vector}, we build the following commutative diagram, starting with the third row:
	\begin{equation}
		\label{nine_lemma}
		\begin{tikzcd}
			&
			1\ar[d] &
			1\ar[d] &
			1\ar[d,dashed]
			\\
			1\ar[r] &
			\piH\ar[r]\ar[d] &
			\piG\ar[r]\ar[d] &
			\piX\ar[r]\ar[d,dashed] &
			1
			\\
			1\ar[r] & 
			\widetilde H\ar[r]\ar[d] & 
			\widetilde G\ar[r]\ar[d] & 
			V\ar[r]\ar[d,dashed] & 
			1
			\\
			1\ar[r] & 
			H\ar[r]\ar[d] & 
			G\ar[r]\ar[d] & 
			X\ar[r]\ar[d,dashed] & 
			1
			\\
			&1&1&1\rlap.
		\end{tikzcd}
	\end{equation}
	That row defines $X$ as the diffeological quotient $G/H$, where we note that $H$ is normal by \eqref{G_normalizes_h}, and $G$ is connected as closure of $H$ (which is D-connected, hence also connected in the subset topology, as the inclusion $H\hookrightarrow G$ is smooth, hence D-continuous). Next we let $\smash{\widetilde G}:=$ universal covering of $G$, $\smash{\widetilde H}:=$ its integral subgroup with Lie algebra $\LH$, and $V:=\smash{\widetilde G/\widetilde H}$. Then \eqref{Chevalley-Malcev-Iwasawa} says that $\smash{\widetilde H}$ is closed, and $\smash{\widetilde H}$ and $V$ are simply connected. In particular $V$ equals $\LG/\LH$, as the unique simply connected Lie group with that abelian Lie algebra. Next define $\piG:=\Ker(\smash{\widetilde G}\to G)$, $\piH:=\piG\cap\smash{\widetilde H}$, and $\piX := \piG/\piH$. These are discrete in every sense, and the five short exact sequences with solid arrows $\begin{tikzcd}[cramped,sep=scriptsize]{}\ar[r]&{}\end{tikzcd}$ are by construction \emph{D-exact}, i.e., the subgroup and quotient in each have the subset and quotient diffeology \eqref{Diffeology_examples}. Therefore the Nine Lemma of \cite[1.30]{Souriau:1985a} says that the diagram has a unique commutative completion by a sixth \emph{D-exact} sequence $\begin{tikzcd}[cramped,sep=scriptsize]{}\ar[r,dashed]&{}\end{tikzcd}$: in other words, $X$ is also the diffeological quotient $V/\piX$, as claimed. Moreover $H$ dense is equivalent to $\piX$ dense, as both are separately equivalent to $X$ having trivial quotient topology \eqref{closed_subgroup}, and the countability of $\piX$, technically required for a quasitorus \cite{Prato:2001,Iglesias-Zemmour:2021,Karshon:2025}, holds here because $G$ is connected, so its fundamental group $\piG$ \eqref{nine_lemma} is countable by \cite[14.2.10(iv)]{Hilgert:2012a}. 
\end{proof}

\section{Examples}\label{examples}

\begin{exam}
	As a simple instance of (\ref{corollary}\ref{D-discrete}), we can realize $H^\bullet(\mathfrak{so}_3) = \RR\oplus\{0\}\oplus\{0\}\oplus\RR$ as de Rham cohomology of $\SO_3(\RR)/\SO_3(\QQ)$; or we could replace the denominator here by a dense free group on two generators \cite{Tomkowicz:2016}; or do the same in any connected semisimple Lie group \cite{Kuranishi:1951}, always getting $H^\bullet(\LG)$.
\end{exam}
	
\begin{exam}
	In another direction, if $\LG$ is a nilpotent Lie algebra and $G$ a corresponding connected Lie group, then (\ref{corollary}\ref{D-discrete}) holds also for some \emph{uncountable} dense D-discrete $H\subset G$, viz.~the subgroups of Hausdorff dimension $0<d<1$ built in \cite[1.1]{Saxce:2013}; their uncountability and D-discreteness follow from \cite[§2.2]{Falconer:2003}.
\end{exam}

\begin{exam}
	Likewise, (\ref{corollary}\ref{vector}) is of interest already for $V=\RR$, as \emph{all} subgroups $\piX\subsetneq\RR$ are D-discrete (see \cite[p.\,364]{Salzmann:2007} or \cite[Ex.\,124]{Iglesias-Zemmour:2013}) and all except the $a\ZZ$ are dense. Using a Hamel basis, one can prove existence of $\smash{2^{2^{\aleph_0}}}$ different subgroups \cite[p.\,8]{Salzmann:2007}, which however defy classification beyond the ``torsion-free rank 1'' (isomorphic to subgroups of $\QQ$): see \cite[pp.\,331--335]{Rotman:1995} and \cite{Paolini:2024}. No matter, (\ref{corollary}\ref{vector}) says that all of them except $\{0\}$ and $\RR$ will give $H_{\textup{dR}}^\bullet(\RR/\piX) = \BigWedge{}^\bullet \RR = \RR\oplus\RR$.
\end{exam}

\begin{exam}
	Remarkably, when $\piX=\ZZ+\alpha\ZZ$ the above does \emph{not} match the periodic cyclic cohomology found in \cite[Thm 53]{Connes:1985} or \cite[§4.8]{Loday:1986} for a crossed product algebra attached to the quasicircle $X=\RR/\piX$. Instead they find $\BigWedge^\bullet\RR^2 = \RR\oplus\RR^2\oplus\RR$, which happens to match the diffeological \emph{Čech} cohomology $\check H^\bullet(X,\RR)$ that \cite{Iglesias-Zemmour:2024a} defined in general and computed for all $V/\piX$ in (\ref{corollary}\ref{vector}). His result is
	\begin{equation}
		\label{Cech}
		\check H^\bullet(V/\piX,\RR) = H^\bullet(\piX,\RR),
	\end{equation}
	the real cohomology of the abstract group $\piX$: in other words, $V/\piX$ behaves like a diffeological $K(\piX,1)$ space. For $\piX=\ZZ+\alpha\ZZ\cong\ZZ^2$ one gets virtually by definition \cite[§II.4, Ex.\,4]{Brown:1982} the real cohomology of a 2-torus, i.e.~the same $\BigWedge^\bullet\RR^2$ as above. (More generally \cite[V.6.4(ii)]{Brown:1982} computes \eqref{Cech}, for all torsion-free abelian $\piX$, as $\operatorname{Hom}_\ZZ(\BigWedge^\bullet \piX,\RR)$ where $\BigWedge^\bullet \piX$ is the exterior algebra of $\piX$ viewed as a $\ZZ$-module.)
\end{exam}

\begin{exam}
	Needless to say, \eqref{main_theorem} admits more examples where $H$ is neither D-connected nor D-discrete. Perhaps the simplest obtains if we replace the subgroup \eqref{Kronecker} by
	\begin{equation}
		\label{Kronecker++}
		H=
		\left\{
		\begin{pmatrix}
			\e{2\pi\i t}&0\\
			0&\pm\e{2\pi\i\alpha t}
		\end{pmatrix}:t\in \RR
		\right\}
		= H^+\sqcup H^-
	\end{equation}
	which has two D-components, yet is connected in the 2-torus because each D-com\-po\-nent is already dense (see e.g.~\cite[9.6]{James:1987}). We note that existence of this \emph{connected yet not arcwise connected subgroup} answers the question left open at the end of \cite[§6.14]{Godement:1982}. Of course the added D-component changes neither the Lie algebra~$\LH$ (nonzero, so $H$ isn't D-discrete) nor $\LG/\LH$, so \eqref{isomorphism} still gives $H_{\textup{dR}}^\bullet(G/H) = \BigWedge^\bullet\RR$. (More generally we could replace $\pm$ by any proper, hence D-discrete, subgroup $\Sigma\subsetneq S^1$ of the unit circle, and still get $\BigWedge^\bullet\RR$.)
\end{exam}	

\appendix
\section{Some subgroup properties}\label{subgroups}

Throughout this appendix, $G$ is a Lie group (always real, finite-dimensional) and $H$ an arbitrary subgroup. We collect five properties which are known, but perhaps not so well known as to be cited without explanation.

\begin{prop}
	\label{initial_subgroup_theorem}
	$H$ always admits a unique manifold structure such that \textup{1$^\circ$)} $i:H\hookrightarrow G$ is an immersion and \textup{2$^\circ$)} a map $F$ from any manifold to $H$ is $C^\infty$ iff $i\circ F$ is $C^\infty$. With this structure\textup, $H$ is a Lie group with Lie algebra \eqref{subalgebra}.
\end{prop}

\begin{proof}[References\textup:]
	This is \cite[III.4.5, Prop.\,9]{Bourbaki:1972} and, to our knowledge, exposed in only three other places: \cite[§6.14]{Godement:1982}, \cite[§§2.2--2.3]{Rossmann:2002} for matrix groups, and \cite[§9.6.2]{Hilgert:2012a} for the exact $C^\infty$ version we quote. In Bourbaki, finding \eqref{subalgebra} requires reading also III.6.4 (Corollary 2), III.6.2 (Example 2), and III.1.1 (Corollary).
\end{proof}

In the language of \eqref{Diffeology_examples}, this says: $(H,$ subset diffeology) \emph{is a manifold}. As stated before \eqref{corollary}, a prefix `D-' shall qualify everything involving the resulting topology, as opposed to the less useful subset topology:

\begin{prop}
	\label{closed_subgroup}
	The following are equivalent\textup:
	\begin{enumerate}[topsep=.5ex,itemsep=.3ex]

		\item
		\label{Lie}
		the subset topology of $H$ is a Lie group topology\textup; 
		
		\item
		\label{Hausdorff}
		the quotient topology of $G/H$ is Hausdorff\textup;
		
		\item
		\label{closed}
		$H$ is closed in $G$\textup.
	\end{enumerate}
	Moreover\textup, the quotient topology is trivial iff $H$ is dense in $G$.
\end{prop}

\begin{proof}[References\textup:]
	(\ref{Lie}) $\Leftrightarrow$ (\ref{closed}) is the closed subgroup theorem 
	\cite{
	Bourbaki:1972,		
	Godement:1982,    
	Hilgert:2012a}		
	and the fact that $H$'s subset topology is not locally compact unless $H$ is closed (e.g.~\cite[§1.1]{Godement:1982} or \cite[9.3.9]{Hilgert:2012a}); (\ref{Hausdorff}) $\Leftrightarrow$ (\ref{closed}) and the last statement are in e.g.~\cite[3.16 and 6.14]{James:1987}.
\end{proof}

\begin{prop}
	\label{Yamabe}
	The D-connected components of $H$ are its arc components.
\end{prop}

\begin{proof}[References\textup:]
	This is Yamabe's theorem \cite{Yamabe:1950} as applied in \cite[9.6.13]{Hilgert:2012a}. (Again \mbox{`D-'} means connected components in the manifold topology, while arc components (unqualified) are in the subset topology; \eqref{Kronecker++} illustrates the nuance.)
\end{proof}

\begin{prop}
	\label{vanEst}
	If $H$ \textup(resp.~its D-identity component $H\o$\textup) is dense in $G$\textup, then
	\begin{enumerate}[topsep=.5ex,itemsep=.3ex]
		
		\item
		\label{invariant}
		$\LH$ is a $G$-invariant ideal\textup: $\Ad(g)(\LH)=\LH$ for all $g\in G$\textup; resp.
		
		\item
		\label{abelian}
		$\LH$ is a $G$-invariant ideal\textup, and $\LG/\LH$ is abelian.
	\end{enumerate}
\end{prop}

\begin{proof}
	(\ref{invariant}) is from \cite[Lemma 1]{Yosida:1937a}, and nowadays proved simply by observing that the normalizer $N_G(\LH)$ is a closed subgroup (e.g.~\cite[III.9.4]{Bourbaki:1972} or \cite[11.1.1]{Hilgert:2012a})  containing $H$, hence equal to $G$ by density.
	
	(\ref{abelian}) is from \cite[1.4.1]{Est:1951}, also found in \cite[III.9.2, Prop.\,5]{Bourbaki:1972} or \cite[§6.15(16)]{Godement:1982}. However, their larger context obscures the simplicity of this \emph{direct argument}: in
	\begin{equation*}
		\begin{tikzcd}[column sep=scriptsize]
			0\ar[r] &
			\LH\ar[r] &
			\LG\ar[r] &
			\LG/\LH\ar[r] &
			0\rlap,
		\end{tikzcd}
	\end{equation*}
	the $G$-invariance of $\LH$ implies that the adjoint actions ($\Ad$ of $G$, $\ad$ of $\LG$) on $\LG$ induce actions $\underline{\Ad}$ and $\underline{\ad}$ on $\LG/\LH$. Now $[\LH,\LG]\subset\LH$ implies $\underline{\ad}(\LH)=0$, hence $\underline{\Ad}(H\o)=\id$, hence by density $\underline{\Ad}(G)=\id$, hence $\underline{\ad}(\LG)=0$ which means $[\LG/\LH,\LG/\LH]=0$.
\end{proof}

\begin{prop}
	\label{Chevalley-Malcev-Iwasawa}
	If $H$ is normal and D-connected in a simply connected $G$\textup, then
	\begin{enumerate*}
		\item
		\label{normal_closed} 
		$H$ is closed\textup, 
		\item
		\label{simply_connected}
		both $H$ and $G/H$ are simply connected.
	\end{enumerate*}
\end{prop}

\begin{proof}[References\textup:]
	(\ref{normal_closed}) is from \cite[p.\,127]{Chevalley:1946}; (\ref{simply_connected}),~while proved in \cite[11.1.21]{Hilgert:2012a} and attributed to Mal'cev--Iwasawa \cite{Samelson:1952}, is perhaps best understood as consequence of the vanishing first, third and fifth terms in the homotopy exact sequence
	\begin{equation*}
		\begin{tikzcd}[column sep=scriptsize]
			\pi_2(G/H)\ar[r] &
			\pi_1(H)\ar[r] &
			\pi_1(G)\ar[r] &
			\pi_1(G/H)\ar[r] &
			\pi_0(H)\rlap.
		\end{tikzcd}\qedhere
	\end{equation*}
\end{proof}

\section{Proving the Chevalley-Eilenberg formula (\ref{Chevalley-Eilenberg})}
\label{proof_CE}

Samelson \cite{Samelson:1952} described \eqref{Chevalley-Eilenberg} as `one of a series of algebraic coboundary formulae, which recently have become popular'; Feigin and Fuchs \cite{Feigin:2000} call it no less than `cumbersome', `unnatural', and `tedious'. Excepting \cite[pp.\,156--161]{Knapp:1988}, all (ten) proofs we could find 
replace Chevalley and Eilenberg's (an opaque induction) by an appeal to Palais' formula \cite[Lemma 1]{Palais:1954}. This feels heuristically backwards, insofar as \eqref{Chevalley-Eilenberg} informed Palais; so there may perhaps be interest in the following simple proof.

We consider the $\LG$-valued 1-form $\Theta_G(\delta g) = g\inv.\delta g$, where $\delta g\in T_gG$ and notation is as explained before \eqref{little_computation}. Extending \eqref{ordinary_d} and later \eqref{pull_back_commutes} to vector-valued forms in the obvious way, we have first (cf.~\cite{Burkhardt:1913,Cartan:1930a,Bourbaki:1972}):

\begin{lemm}[Maurer--Cartan]
	\label{Maurer-Cartan}
	$d\Theta_G(\delta g,\delta'g)=[\Theta_G(\delta'g),\Theta_G(\delta g)]$.
\end{lemm}
 
\begin{proof}
	1.~Suppose $G=\GL_n(\RR)$. Then we are in an open set in $\RR^{n\times n}$, and $g\inv.\delta g$ is literally a matrix product and $[\cdot,\cdot]$ the commutator. So definition \eqref{ordinary_d} applies and gives
	\begin{equation}
		d\Theta_G(\delta g,\delta'g) =
		\smash[t]{\frac{\partial g\inv}{\partial g}(\delta g).\delta'g} -
		\smash[t]{\frac{\partial g\inv}{\partial g}(\delta'g).\delta g} =
		[g\inv.\delta'g,g\inv.\delta g]
	\end{equation}
	as claimed, where we used the `freshman formula' $(\partial g\inv/\partial g)(\delta g) = -g\inv.\delta g.g\inv$. In fact, since both sides of \eqref{Maurer-Cartan} are left-invariant, it would have been enough (and will suffice in the rest of this proof) to do the calculation at $g=e$ and get
	\begin{equation}
		\label{at_e}
		d\Theta_G(Z_0,Z_1)=[Z_1,Z_0]
		\rlap{\qquad\quad $\forall\,Z_0, Z_1\in\LG$.}
	\end{equation}

	2.~Suppose $R:G\to H$ is a morphism of Lie groups with differential $r:\LG\to\LH$. Then the $\Theta$s are related by $R^*\Theta_H = r\circ\Theta_G$. Therefore $R^*d\Theta_H = r\circ d\Theta_G$, i.e., $d\Theta_H(r(Z_0),r(Z_1)) = r(d\Theta_G(Z_0,Z_1))$. From this one readily deduces that \eqref{at_e} for $H$ implies it for $G$ if $r$ is one-to-one, and conversely if $r$ is onto.
	
	3.~Suppose $G$ is simply connected. By Ado's theorem, we have a morphism $R:G\to\GL_n(\RR)$ with injective differential; so steps 1 and 2 imply \eqref{at_e} for $G$.
	
	4.~Suppose $G$ is arbitrary, and let $\smash{\widetilde G}\to G^{\mathrm o}$ be the universal covering of its identity component. Then the composition $R:\smash{\widetilde G}\to G^{\mathrm o}\hookrightarrow G$ has surjective differential, so steps 2 and 3 imply \eqref{at_e} for $G$. 
\end{proof}

\begin{lemm}
	\label{shuffle_lemma}
	Let $\alpha$ be a $2$-form and $\beta$ a $(k-1)$-form. Then
	\begin{equation*}
		(\alpha\wedge\beta)(Z_0,\dots,Z_k) =
		\!\!\sum_{0\leqslant i<j\leqslant k}\!\!
		(-1)^{i+j-1}\alpha(Z_i,Z_j)
		\beta(Z_0,\twoldots,\widehat{Z}_i,\twoldots,\widehat{Z}_j,\twoldots,Z_k).
	\end{equation*}
\end{lemm}

\begin{proof}
	The left-hand side is by definition $\sum_\sigma(-1)^\sigma\alpha(Z_{\sigma(0)},Z_{\sigma(1)}) \beta(Z_{\sigma(2)},\dots,Z_{\sigma(k)})$ where the sum is over permutations of $\{0,\dots,k\}$ that are increasing over $\{0,1\}$ and $\{2,\dots,k\}$: see e.g.~\cite[p.\,260]{Bredon:1993}. Such a $\sigma$ is determined by $(i,j)=(\sigma(0),\sigma(1))$, and its sign $(-1)^\sigma$ equals $(-1)^{i+j-1}$, as one sees by counting $i+j-1$ crossings in
	\begin{equation*}
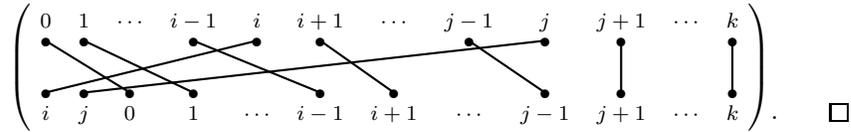
	\begin{gathered}[b]
	{\footnotesize
		\left(
		\begin{tikzcd}[column sep=0.4em, arrows={dash,thick,shorten=-2.2mm},baseline=-.7mm]
		0
		& 1
		& \cdots
		& i-1
		& i
		& i+1
		& \cdots
		& j-1
		& j
		& j+1
		& \cdots
		& k
		\\[-5.5ex]
		  \bullet   \ar[drr]
		& \bullet   \ar[drr]
		&
		& \bullet   \ar[drr]
		& \bullet   \ar[dllll]
		& \bullet   \ar[dr]
		& 
		& \bullet   \ar[dr]
		& \bullet   \ar[dlllllll]
		& \bullet   \ar[d]
		& 
		& \bullet   \ar[d]
		\\[-1.5ex]
		  \bullet
		& \bullet
		& \bullet
		& \bullet
		& 
		& \bullet
		& \bullet 
		& 
		& \bullet
		& \bullet
		& 
		& \bullet
		\\[-5.5ex]
		i 
		& j 
		& 0 
		& 1 
		& \cdots 
		& i-1 
		& i+1
		& \cdots 
		& j-1 
		& j+1 
		& \cdots 
		& k
		\end{tikzcd}
		\right)
		}
		\raisebox{-4ex}{.}
	\\\vspace{-4ex}
	\end{gathered}
	\qedhere
	\end{equation*}
\end{proof}

\Needspace*{5\baselineskip}
\begin{proof}[Proof of \textup{(\ref{Chevalley-Eilenberg})}.]
	Since $d$ is linear, and products $\omega = \theta_1\wedge\dots\wedge\theta_k$ of invariant $1$-forms span $\Omega^k(G)^G$, we can assume without loss of generality that $\omega$ is such a product. Then e.g.~\cite[pp.\,260, 262]{Bredon:1993} gives the formulas $\omega(Z_1,\dots,Z_k) = \smash{\det(\theta_i(Z_j))_{i,j=1}^k}$ and (graded Leibniz) $d\omega = \smash{\sum_{m=1}^k}(-1)^{m+1}\theta_1\wedge\dots\wedge d\theta_m\wedge\dots\wedge\theta_k$. As $2$-forms wedge-commute with everything (repeat the proof of \eqref{shuffle_lemma} with $\alpha$ and $\beta$ switched), this last product equals $d\theta_m\wedge\beta_m$ where $\beta_m := \theta_1\wedge\dots\wedge\smash{\widehat{\theta}_m}\wedge\dots\wedge\theta_k$. Thus we see that the left-hand side of \eqref{Chevalley-Eilenberg} equals
	\begin{align*}
		\sum_{m=1}^k&(-1)^{m+1}
		(d\theta_m\wedge\beta_m)(Z_0,\dots,Z_k)
		\\
		& = \sum_{m=1}^k(-1)^{m+1}
		\!\!\sum_{0\leqslant i<j\leqslant k}\!\!
		(-1)^{i+j-1}d\theta_m(Z_i,Z_j)
		\beta_m(Z_0,\twoldots,\widehat Z_i,\twoldots,\widehat Z_j,\twoldots,Z_k)
		\\
		& = \sum_{0\leqslant i< j\leqslant k}(-1)^{i+j}
		\sum_{m=1}^k(-1)^{m+1}\theta_m([Z_i, Z_j])
		\beta_m(Z_0,\twoldots,\widehat Z_i,\twoldots,\widehat Z_j,\twoldots,Z_k)
		\\
		& = \sum_{0\leqslant i<j\leqslant k}(-1)^{i+j}\det
		\left(\begin{smallmatrix}
		    \theta_1([Z_i, Z_j]) & \theta_1(Z_0) & {\textstyle\cdots} & 
			\widehat{\theta_1(Z_i)} & {\textstyle\cdots} &
			\widehat{\theta_1(Z_j)} & {\textstyle\cdots} &
			\theta_1(Z_k)
			\\[-2pt]
		    \vdots & \vdots & & \vdots & & \vdots & & \vdots
			\\[2pt]
		    \theta_k([Z_i, Z_j]) & \theta_k(Z_0) & {\textstyle\cdots} & 
			\widehat{\theta_k(Z_i)} & {\textstyle\cdots} &
			\widehat{\theta_k(Z_j)} & {\textstyle\cdots} &
			\theta_k(Z_k)
		\end{smallmatrix}\right)
		\\[1.5ex]
		& = \sum_{0\leqslant i<j\leqslant k} (-1)^{i+j}
		\omega([Z_i,Z_j],Z_0,\twoldots,\widehat Z_i,\twoldots,\widehat Z_j,\twoldots,Z_k)
	\end{align*}
	as claimed. Here the first equality is by \eqref{shuffle_lemma}; the second is by \eqref{at_e} which gives $d\theta(Z_0,Z_1) = -\theta([Z_0,Z_1])$ for any left-invariant $1$-form $\theta = \<\mu,\Theta_G(\,\cdot\,)\>$ ($\mu\in\LG^*$); and the third is cofactor expansion of the determinant along its first column.
\end{proof}

\begin{rema}
	The above argument, proving \eqref{Chevalley-Eilenberg} from its case $k=1$ and~the graded Leibniz property, can be reversed to show that $\dCE$, when \emph{defined} on $\BigWedge{}^\bullet\LG^*$ by \eqref{Chevalley-Eilenberg}, has the graded Leibniz property. This a key step that \cite{Chevalley:1948} left to the reader, in their algebraic proof that $\dCE^2=0$. (See their §14, p.\,105, case $p=1$.)
\end{rema}

\let\i\dotlessi
\let\l\polishl
\let\o\norwegiano
\let\u\russianbreve

\bibliographystyle{amsalpha}

\end{document}